\documentclass[12pt, reqno]{amsart}
\usepackage{amsmath, amsthm, amscd, amsfonts, amssymb, graphicx, color}
\usepackage[bookmarksnumbered, colorlinks, plainpages]{hyperref}
\hypersetup{colorlinks=true,linkcolor=red, anchorcolor=green, citecolor=cyan, urlcolor=red, filecolor=magenta, pdftoolbar=true}

\newtheorem{theorem}{Theorem}[section]

\theoremstyle{definition}

\theoremstyle{remark}

\numberwithin{equation}{section}

\begin{document}

\setcounter{page}{1}

\title[Global existence and blow-up ...]{Global existence and blow-up for a space and time nonlocal reaction-diffusion equation}

\author[A. Alsaedi, M. Kirane \MakeLowercase{and} B. T. Torebek]{Ahmed Alsaedi, Mokhtar Kirane$^{*}$ \MakeLowercase{and} Berikbol T. Torebek}

\address{\textcolor[rgb]{0.00,0.00,0.84}{Ahmed Alsaedi \newline NAAM Research Group, Department of Mathematics, \newline Faculty of Science, King Abdulaziz University, \newline P.O. Box 80203, Jeddah 21589, Saudi Arabia}}
\email{\textcolor[rgb]{0.00,0.00,0.84}{aalsaedi@hotmail.com}}

\address{\textcolor[rgb]{0.00,0.00,0.84}{Mokhtar Kirane\newline LaSIE, Facult\'{e} des Sciences, \newline Pole Sciences et Technologies, Universit\'{e} de La Rochelle \newline Avenue M. Crepeau, 17042 La Rochelle Cedex, France \newline NAAM Research Group, Department of Mathematics, \newline Faculty of Science, King Abdulaziz University, \newline P.O. Box 80203, Jeddah 21589, Saudi Arabia}}
\email{\textcolor[rgb]{0.00,0.00,0.84}{mkirane@univ-lr.fr}}

\address{\textcolor[rgb]{0.00,0.00,0.84}{Berikbol T. Torebek \newline Al--Farabi Kazakh National University \newline Al--Farabi ave. 71, 050040, Almaty, Kazakhstan \newline Institute of
Mathematics and Mathematical Modeling \newline 125 Pushkin str.,
050010 Almaty, Kazakhstan \newline Department of Mathematics: Analysis, Logic and Discrete Mathematics \newline Ghent University \newline Krijgslaan 281, Building S8, B 9000 Ghent, Belgium}}
\email{\textcolor[rgb]{0.00,0.00,0.84}{berikbol.torebek@ugent.be}}


\let\thefootnote\relax\footnote{$^{*}$Corresponding author}



\begin{abstract} A time-space fractional reaction-diffusion equation in a bounded domain is considered. Under some conditions on the initial data, we show that
solutions may experience blow-up in a finite time. However, for realistic initial conditions, solutions are global in time. Moreover, the asymptotic behavior of bounded solutions is analysed.
\end{abstract} \maketitle
\section{Introduction}
In this paper, we consider the time and space fractional equation
\begin{equation}\label{1}\partial^\alpha_tu+(-\Delta)_{\Omega}^su=-u(1-u),\, x\in\Omega, t>0,
\end{equation}
$\Omega \subset \mathbb{R}^{n}$ is a bounded domain, supplemented with the homogeneous boundary condition
\begin{equation}\label{2}u=0,\,\,x\in \mathbb{R}^n\setminus\Omega,\,t>0,
\end{equation}
and with the initial condition
\begin{equation}\label{3}u(x,0)=u_0(x),\,\,x\in\Omega,
\end{equation}
where $\partial_t^\alpha$ is the Caputo fractional derivative (see \cite[page 90]{Kilbas})
\[
\partial_t^\alpha u(x,t)=\frac{1}{\Gamma(1-\alpha)}\int\limits_0^t(t-\tau)^{-\alpha}\partial_\tau u(x,\tau)d\tau,
\]
of order $\alpha\in (0,1],$ defined for a differentiable function $u,$ and $\left(-\Delta\right)_{\Omega}^s$ is the regional fractional Laplacian
\begin{align*}
\left(-\Delta\right)_{\Omega}^s u(x,t)&=C_{n,s}\, \textrm{P.V.} \int\limits_{\Omega}\frac{u(x,t)-u(\xi,t)}{|x-\xi|^{n+2s}} \, d\xi
\\&= \displaystyle\lim_{\varepsilon\downarrow 0}\left(-\Delta\right)_{\Omega, \varepsilon}^su(x,t), x \in \Omega,\end{align*}
defined for  $u(x,t)$ such that
\[
\int_{\Omega}\frac{\vert u(x,t)\vert}{(1+\vert x\vert)^{n+2s}}\, dx < \infty,
\]
where $s\in (0,1)$ and $C_{n,s}$ is a normalizing constant (whose value is not important here), and
\[
\left(-\Delta\right)_{\Omega,\varepsilon}^s u(x,t) = C_{n,s}\,
\int_{\left\lbrace  \xi \in \Omega : |x-\xi| > \varepsilon \right\rbrace } \frac{u(x,t)-u(\xi,t)}{|x-\xi|^{n+2s}} \, d\xi, \; x \in \mathbb{R}^n.
\]
The domain of the operator $(-\Delta)^s_\Omega$ is
$$D((-\Delta)^s_\Omega)=\{u\in H^s_0(\bar\Omega),\, (-\Delta)^s_\Omega u\in L^2(\Omega)\}.$$
Set $\mathcal{L}^\infty(\Omega):=\overline{D(A_\infty)}^{L^\infty(\Omega)},\, A_\infty=(-\Delta)^s_\Omega,$ where
$$D(A_\infty)=\{u\in D((-\Delta)^s_\Omega)\cap L^\infty(\Omega);\, Au\in L^\infty(\Omega)\},\, A_\infty=Au.$$
The fractional Sobolev space $H^s(\Omega)$ is defined by \cite[Section 2]{AB17}
$$H^s(\Omega)=\{u\in L^2(\Omega):\, |u|_{H^s(\Omega)}<\infty\},$$
where
\[
|u|_{H^s(\Omega)}=\left(\int\int\limits_{\Omega\times\Omega}\frac{|u(x)-u(\xi)|^2}{|x-\xi|^{n+2s}}dxd\xi\right)^{\frac{1}{2}}
\]
denotes the Gagliardo (Aronszajn-Slobodeckij) seminorm. The space $H^s(\Omega)$ is a Hilbert space; it is  equipped with the norm
\[
\|u\|_{H^s(\Omega)}=\|u\|_{L^2(\Omega)}+|u|_{H^s(\Omega)}.
\]
The space $H_0^s(\bar\Omega)$ consists of functions of $H^s(\bar\Omega)$ vanishing on the boundary of $\Omega.$

Our paper is motivated by the recent one of \cite{AAAKT15} in which global solutions and blowing-up solutions to the equation \eqref{1}, when $s=1$
are proved. Our problem \eqref{1}-\eqref{3} is a natural generalization of results in \cite{AAAKT15}. We will prove the existence of globally bounded solutions as well as blowing-up solutions under some condition
imposed on the initial data. Note that, similar studies for time-fractional and time-space-fractional reaction-diffusion equations were considered in \cite{AA18, AGKW19, DCCM15, CSWSS18, KSZ17, KLW16, VZ15, VZ17, Gal}.

We will use first eigenfunction $e_1(x) > 0$ associated to the first eigenvalue $\lambda_1 >0$ that satisfies the fractional eigenvalue problem \cite[Theorem 2.8]{Bras}
\begin{equation}\label{9}\begin{split}
                           &(-\Delta)_\Omega^se_n(x)=\lambda_n e_n(x), \,\,\,x\in\Omega,\\
                            & e_n(x)=0, \,\,\,\,\,\,\,\,\,\,\,\,\,\,\,\,\,\,\,\,\,\,\,\,\,x\in \mathbb{R}^n\setminus\Omega,
                         \end{split}
\end{equation} is normalized  such that $\displaystyle\int\limits_\Omega e_1(x)dx=1.$

The space $C^{0,\kappa}((0,T), L^\infty(\Omega))$ consists of functions that are continuous and $\kappa$-H\"{o}lderian with respect to $t.$
\section{Main results}
\begin{theorem}\begin{description}
  \item[(i)] Let $u_0(x)\in \mathcal{L}^\infty(\Omega)$ be such that $0\leq u_0(x)\leq 1.$ Then, problem \eqref{1}-\eqref{3} admits a global strong solution on $[T_0,\infty),$ for every $T_0>0,$ i.e., $u$ is a mild solution, and $$u\in C^{0,\kappa}((0,\infty), L^\infty(\Omega)),\,\,\,\textrm{for some} \,\,\,\kappa>0;$$ $$u(.,t)\in D(A_\infty),\,\,\,\textrm{for all}\,\,\,t>0,\,\,\,\textrm{and}\,\,\,\, \partial_t^\alpha u\in C((0,\infty), L^\infty(\Omega)).$$ The equation \eqref{1} is satisfied for $t>0.$ It satisfies $0\leq u(x,t)\leq 1$ for all $x\in \Omega,\,t>0$ and 
      \begin{equation}\label{7}\|u(t,\cdot)\|^2_{L^2(\Omega)}\leq \frac{C}{1+\nu t^\alpha}\|u_0\|^2_{L^2(\Omega)} \leq Ct^{-\alpha},\,\,\,t>0,\end{equation} for some positive constant $C.$
  \item[(ii)] If $1+\lambda_1\leq \displaystyle\int\limits_{\Omega}u_0(x)e_1(x)dx=H_0,$ then the solution of problem \eqref{1}-\eqref{3} blows-up in a finite time $T^*$ that satisfies the bi-lateral estimate
\begin{equation*}\left(\frac{\Gamma(\alpha+1)}{4(H_0+1/2)}\right)^{\frac{1}{\alpha}}\leq T^*\leq \left(\frac{\Gamma(\alpha+1)}{H_0}\right)^{\frac{1}{\alpha}}.\end{equation*}
\end{description}
 \end{theorem}
\begin{proof} The proof  is divided in steps.\\
\textbf{Global existence.} Firstly, we show that $u\ge 0.$ Multiplying scalarly in ${{L}^{2}}(\Omega)$ equation \eqref{1} by $\tilde{u}:=\min \left( u,0 \right),$ we obtain
\begin{align*}\int\limits_{\Omega}\partial_{t}^{\alpha }\tilde{u}\cdot\tilde{u}dx+\int\limits_{\Omega}(-\Delta)_\Omega^s\tilde{u}\cdot\tilde{u}\, dx=\int\limits_{\Omega}\tilde{u}^2\left(\tilde{u}-1 \right) \, dx.
\end{align*}
Using the estimate \cite[Proposition 2.4]{AB17}
\begin{equation}\label{*}0\leq\|u\|^2_{L^2(\Omega)}\lesssim \int\limits_{\Omega}\int\limits_{\Omega}\frac{|u(x)-u(y)|^2}{|x-y|^{n+2s}} \, dxdy\lesssim \left((-\Delta)_\Omega^su, u\right), \,u\in H^s(\Omega),
\end{equation}
\noindent follows
\[
-\int\limits_{\Omega}(-\Delta)_\Omega^s\tilde{u}\cdot\tilde{u} \, dx\leq 0.
\]
Using the inequality \cite[Remark 3.1]{AAK17}
\begin{equation*}
2w(t)\partial^\alpha w(t)\geq\partial^\alpha w^2(t),\,w\in C^1(0,\infty),\end{equation*}
via an approximating process, we obtain
\begin{equation}\label{5}\partial_{t}^{\alpha }\int\limits_{\Omega}\tilde{u}^2\left( x,t \right)dx\lesssim \int\limits_{\Omega}\tilde{u}^{2}\left( x,t \right)dx.
\end{equation}

Setting  $\varphi(t)=\int\limits_{\Omega}\tilde{u}^{2}\left( x,t \right)dx$ in \eqref{5}, we get
\begin{equation*}\left\{\begin{array}{l}
\partial^{\alpha }\varphi(t)\lesssim\varphi(t),\\
\varphi(0)=0,\end{array}\right.\end{equation*}
which implies $\displaystyle\int\limits_{\Omega}\tilde{u}^{2}\left( x,t \right)dx=0$ as $\tilde{u}\left( x,0 \right)=0$, hence $\tilde{u}=0$; whereupon $u\geq 0.$

Now, we show the upper estimate $u\le 1.$ Multiplying scalarly in ${{L}^{2}}(\Omega)$ equation \eqref{1} by $\hat{u}:=\min \left( 1-u,0 \right),$ we obtain
\begin{align*}\int\limits_{\Omega}\partial_{t}^{\alpha }\hat{u}\cdot\hat{u}dx+\int\limits_{\Omega}(-\Delta)_\Omega^s\hat{u}\cdot\hat{u}dx=\int\limits_{\Omega}\hat{u}^2\left(\hat{u}-1 \right)dx.\end{align*}
Repeating the above calculations for $\hat{u}:=\min \left( 1-u,0 \right)$, we get
\begin{equation*}\partial_{t}^{\alpha }\int\limits_{\Omega}\hat{u}^2\left( x,t \right)dx\lesssim\int\limits_{\Omega}\hat{u}^{2}\left( x,t \right)dx.
\end{equation*}
This implies that $\int\limits_{\Omega}\hat{u}^{2}\left( x,t \right)dx=0$ as $\hat{u}\left( x,0 \right)=0$, hence $\hat{u}=0$; whereupon $u\leq 1.$ The result follows as $0\le u\le 1.$ Consequently, the solution exists on $[0, +\infty)$.\\

\textbf{Decay estimate.} Since $0 \leq u\leq 1$, then  $-u + u^2 \leq 0.$
So, $u$ satisfies to the equation
\begin{equation}\label{8}
\partial^\alpha_tu+(-\Delta)_\Omega^su\leq 0, x\in \Omega,\,t>0,
\end{equation}
with initial and boundary conditions \eqref{2}, \eqref{3}.

Multiplying scalarly equation \eqref{8} by $u$ and using Poincar\'{e}-type inequality \eqref{*}, we obtain
\begin{align*}\partial^\alpha E(t)+\nu E(t)\leq 0,\,\,t>0,\end{align*} where we have set $E(t)=\displaystyle\int\limits_{\Omega}u^2dx.$

By comparison, we have $E(t)\leq \bar{E}(t),$ where $\bar{E}(t)$ is the solution of the problem
\begin{equation*}\partial^\alpha\bar{E}(t)+\nu\bar{E}(t)=0,\,t>0,
\end{equation*}
\begin{equation*}\bar{E}(0)=E_0=\int\limits_{\Omega}u^2_0(x)dx,
\end{equation*}
whose unique solution is
\[
\bar{E}(t)=E_0 E_{\alpha}(-\nu t^\alpha),
\]
where $E_{\alpha}(z)$ is the Mittag-Leffler function \cite[page 40]{Kilbas}:
\[
E_{\alpha}(z)=\sum\limits_{m=0}^\infty\frac{z^m}{\Gamma(\alpha m+1)}.
\]
The following Mittag-Leffler function's estimate is known \cite[Theorem 4]{Simon}:
\[
E_{\alpha}(-z)\lesssim \frac{1}{1+z},\,\,z>0.
\]
Then by comparison principle \cite{JiaLi} with the equation \begin{equation*}
\partial^\alpha_t\bar{u}+(-\Delta)_\Omega^s\bar{u}=0, x\in \Omega,\,t>0,
\end{equation*}
with initial-boundary conditions \eqref{2}, \eqref{3}, the estimate \eqref{7} then follows.\\

\textbf{Blow-up.} Multiplying equations \eqref{1} by $e_1(x)$ and integrating over $\Omega,$ we obtain
\begin{equation}\label{10}\begin{split}\partial^\alpha\int\limits_{\Omega}u(x,t)e_1(x)dx& +\int\limits_{\Omega}\left(-\Delta\right)_\Omega^su(x,t)e_1(x)dx \\&=\int\limits_{\Omega}u(x,t)\left(u(x,t)-1\right)e_1(x)dx.\end{split}
\end{equation}
Since
\begin{align*}
\int\limits_{\Omega}\left(-\Delta\right)_\Omega^su(x,t)e_1(x)dx&=\int\limits_{\Omega}u(x,t)\left(-\Delta\right)_\Omega^se_1(x)dx \\&=\lambda_1\int\limits_{\Omega}u(x,t)e_1(x) \,dx,\end{align*}
as $u=0,\,e_1(x)=0,\,\,\,x\in \mathbb{R}^n\setminus\Omega$ and
\[
\left( \int\limits_\Omega u(x,t)e_1(x) \,dx\right)^{2}  \leq \int\limits_\Omega u^2(x,t)e_1(x) \,dx,
\]
the function $H(t) = \displaystyle\int\limits_\Omega u(x,t)e_1(x) \,dx $ then satisfies
\begin{equation}\label{11}\partial^\alpha H(t)+(1+\lambda_1) H(t)\geq H^2(t).
\end{equation}
Let $\tilde{H}(t)=H(t)-(1+\lambda_1),$ then from \eqref{11}, we get
\begin{equation}\label{12}\partial^\alpha \tilde{H}(t)\geq \tilde{H}(t)\left(\tilde{H}(t)+1+\lambda_1\right)\geq \tilde{H}(t)\left(\tilde{H}(t)+1\right).
\end{equation}
As $0\leq \tilde{H}_0=\tilde{H}(0),$ then from the results in \cite[Formula (8)]{DFR14}, the solution of inequality \eqref{12} blows-up in a finite time.\\
The proof of the Theorem is complete.
\end{proof}

\section*{Acknowledgements} The research of Alsaedi and Kirane is supported by NAAM research group, University of King Abdulaziz, Jeddah. The research of Torebek is financially supported in parts by the FWO Odysseus 1 grant G.0H94.18N: Analysis and Partial Differential Equations and by a grant No.AP08052046 from the Ministry of Science and Education of the Republic of Kazakhstan. No new data was collected or generated during the course of research

\end{document}